\documentclass[
aps,%
12pt,%
final,%
notitlepage,%
oneside,%
onecolumn,%
nobibnotes,%
nofootinbib,%
superscriptaddress,%
noshowpacs,%
centertags]%
{revtex4}
\usepackage{amsmath, amsthm, amsfonts, amssymb, dsfont, accents}
\usepackage{graphicx,color,srcltx}

\newtheorem{theorem}{Theorem}
\newtheorem{lemma}{Lemma}

\theoremstyle{remark}

\def\Om{\Omega}

\def\e{\varepsilon}
\def\g{\gamma}
\def\G{\Gamma}
\def\l{\lambda}
\def\p{\partial}
\def\D{\Delta}

\def\a{\alpha}
\def\b{\beta}

\def\d{\delta}

\def\z{\zeta}

\def\di{\,d}

\def\iu{\mathrm{i}}

\def\He{\mathcal{H}_\e}
\def\he{\mathfrak{h}_\e}

\def\Ho{\mathcal{H}_0}
\def\ho{\mathfrak{h}_0}

\def\Wo{\mathring{W}_2}

\def\A{\mathcal{A}}

 \DeclareMathOperator{\RE}{Re}
\DeclareMathOperator{\IM}{Im} \DeclareMathOperator{\spec}{\sigma}
\DeclareMathOperator{\discspec}{\sigma_{d}}
\DeclareMathOperator{\essspec}{\sigma_{e}}

\DeclareMathOperator{\Dom}{\mathcal{D}}

\allowdisplaybreaks

\numberwithin{equation}{section}

\begin{document}

\title{Creation of spectral bands \\ for a periodic  domain with small windows}

\author{\firstname{D.I.}~\surname{Borisov}}

\email{BorisovDI@yandex.ru}

\affiliation{Institute of Mathematics CC USC RAS \& Bashkir State Pedagogical University \& University of Hradec Kr\'alov\'e}

\begin{abstract}
We consider a Schr\"odiner operator in a periodic system of strip-like domains coupled by small windows. As the windows close, the domain decouples into an infinite series of identical domains.  The operator similar to the original one but on one copy of these identical domains has an essential spectrum. We show that once there is a virtual level at the threshold of this essential spectrum, the windows turns this virtual level into the spectral bands for the original operator. We study the structure and the asymptotic behavior of these bands.
\end{abstract}

\maketitle

\section{Introduction}

The paper is devoted to a Schr\"odinger operator subject to the Dirichlet condition in a periodic system of unbounded domains coupled by small windows, see Fig. 1.   Once the windows close, the domain decouples in an infinite series of identical domains. The spectrum of the original perturbed operator converges to that of the similar operator but in the aforementioned decoupled domain. This limiting set is the spectrum of the similar operator on the periodicity cell of the decoupled domain.  Each point in the limiting spectrum generates a band in the spectrum of the original operator. And our  aim is to study the structure of such bands.

The models similar to ours were studied in the series of papers \cite{BRT}, \cite{Na3}, \cite{Na5}, \cite{Pan}, \cite{Yo}, \cite{RJMP15}. Under the assumption that the periodicity cell is a bounded domain, in \cite{BRT}, \cite{Na3}, \cite{Na5}, \cite{Pan}, \cite{Yo} there were obtained either estimates for the bands or  asymptotics for the band functions. The asymptotics were obtained under the assumption that the limiting spectral point is a simple isolated eigenvalue of the limiting operator in the periodicity cell. In \cite{RJMP15} the periodicity cell was supposed to be either bounded or unbounded. The limiting spectral point was a multiple isolated eigenvalue. For the band functions converging to this limiting eigenvalue, the asymptotics were obtained. Thanks to these asymptotics, it was found that typically, a multiple limiting eigenvalue generates as many bands as its multiplicity is. It was also shown that for one of the bands, the associated band function attains its maximum or/and minimum inside the Brillouin zone neither at the center nor at the end-points.

\begin{figure}[t]
\includegraphics[scale=0.5]{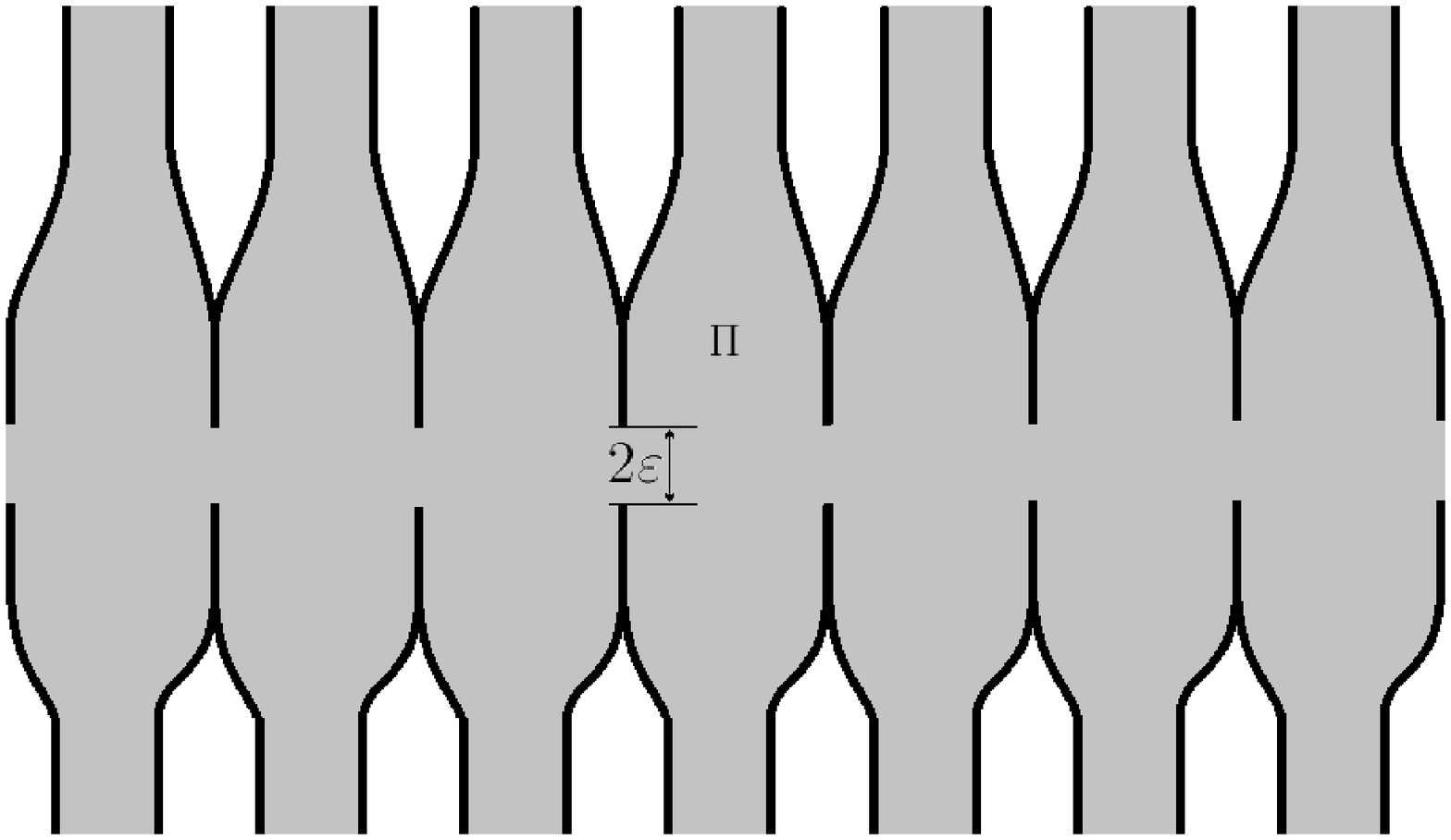}
\caption{Domain}
\end{figure}

The present paper can be regarded as a continuation of \cite{RJMP15}. Here the periodicity cell is assumed to be a strip-like domain. The operator on this domain has an essential spectrum which is invariant w.r.t. the windows. As the windows open, from the threshold of the essential spectrum there can emerge additional band functions generating extra spectral bands for the perturbed operator. We prove that it is indeed the case provided there is a virtual level at the threshold of the limiting operator. We describe the asymptotic behavior of the emerging bands. We also show that if the multiplicity of the virtual level is two, i.e., there are two non-trivial associated resonance solutions, they generate two separate bands in the vicinity of the threshold of the essential spectrum. And the band function associated with one of these bands attains its minimum and/or maximum at the internal points of the Brillouin zone.

This is a completely new phenomenon in comparison with the results in the above cited papers. The main difference is that there the bands shrank to single points which were the isolated eigenvalues of the limiting operator on the periodicity cell. In our case the studied bands disappear in the essential spectrum of the limiting operator on the periodicity cell, i.e., in the limit, there are even no single points. To the best of the author's knowledge, such phenomenon was not studied before.

\section{Problem and main results}

Let $x=(x_1,x_2)$ be Cartesian coordinates in $\mathds{R}^2$, $\Pi$
be an unbounded domain in $\mathds{R}^2$ with $C^3$-boundary possessing the following properties:

\begin{itemize}
\item Domain $\Pi$ lies between the lines $x_1=0$ and $x_1=d$, $d>0$, i.e., $\Pi\subseteq\{x: \, 0<x_1<d\}$.

\item For some $x_2^\infty>0$,, domain $\Pi$ has to straight outlets in the region $|x_2|>x_2^\infty$:
    \begin{equation*}
    \Pi\cap\{x:\, \pm x_2>x_2^\infty\}=\{x:\, a_\pm<x_1<a_\pm+1,\ |x_2|>x_2^\infty\},
    \end{equation*}
    where $a_\pm$  are some fixed constant satisfying $0\leqslant a_\pm<a_\pm+1\leqslant d$.

\item For some $x_2^0>0$, as $|x_2|>x_2^0$, the boundary of $\Pi$ is formed by two vertical segments:
    \begin{equation*}
    \p\Pi\cap\{x:\, |x_2|<x_2^0\}=\{x:\, x_1=0,\, |x_2|<x_2^0\}\cup \{x:\, x_1=d,\, |x_2|<x_2^0\}.
    \end{equation*}
\end{itemize}

By $\Om$ we denote the domain obtained by translations of $\Pi$ along axis $Ox_1$:
\begin{equation*}
\Om:=\bigcup\limits_{k\in\mathds{Z}} \{x:\, (x_1-kd,x_2)\in\Pi\}.
\end{equation*}
This domain is periodic along axis $Ox_1$ and the periodicity cells have common boundaries, a part of them are the vertical segments $\{x:\, x_1=kd,\, |x_2|<x_2^0\}$. We couple these cells by a periodic set of small windows; the obtained domain is denoted by $\Om_\e$:
\begin{equation*}
\Om_\e:=\Om\cup\bigcup\limits_{k\in\mathds{Z}} \{x:\, x_1=kd,\,|x_2|<\e\}.
\end{equation*}

The main object of our study is the Schr\"odinger operator
\begin{equation*}%\l%abel{2.1}
\He:=-\D + V
\end{equation*}
in $\Om_\e$ subject to the Dirichlet condition. Here  potential $V$ is supposed to be infinitely differentiable in $\overline{\Om}$ and $1$-periodic w.r.t. $x_1$. We also assume that  $V(x)\equiv0$ as $|x_2|>x_2^\infty$.

Rigorously we introduce $\He$ as the self-adjoint operator in $L_2(\Om_\e)$ associated with the lower-semibounded closed symmetric sesquilinear form
\begin{equation*}%\l%abel{2.2}
\he(u,v):=\big(\nabla u,\nabla v\big)_{L_2(\Om_\e)} + (Vu,v)_{L_2(\Om_\e)}
\end{equation*}
in $L_2(\Om_\e)$ on the domain $\Wo^1(\Om_\e)$. Hereinafter, for a given domain $Q$, the symbol $\Wo^j(Q)$, $j=1,2$, stands for the subspace of Sobolev space $W_2^j(Q)$ consisting of the functions with zero trace on $\p Q$. For the similar subspace of the functions with zero trace on a curve $S\subset S\subset \overline{Q}$, we shall employ the symbol $\Wo^j(Q,S)$.

In this paper, we study the structure of the spectrum of $\He$ as $\e\to+0$. Our first main result states the norm resolvent convergence for $\He$ as $\e\to+0$, and, in particular, it implies the convergence of the spectrum. The limiting operator is
\begin{equation*}
\Ho:=-\D+V
\end{equation*}
in $\Om$ subject to the Dirichlet condition. It is also introduced as associated with a sesquilinear form; the latter reads as
\begin{equation*}
\ho(u,v):=\big(\nabla u,\nabla v\big)_{L_2(\Om)} + (Vu,v)_{L_2(\Om)}
\end{equation*}
in $L_2(\Om)$ on the domain $\Wo^1(\Om)$.

By $\|\cdot\|_{X\to Y}$ we denote the norm of a bounded operator acting from a Hilbert space $X$ into a Hilbert space $Y$.

Now we are in position to formulate our first main result.

\begin{theorem}\label{th2.1}
For sufficiently small $\e$, the estimate
\begin{equation*}
\|(\He-\iu)^{-1}-(\Ho-\iu)^{-1}\|_{L_2(\Om)\to W_2^1(\Om_\e)} \leqslant C\e|\ln\e|^{\frac{1}{2}},
\end{equation*}
holds true, where $C$ is a constant independent of $\e$, $\iu$ is the imaginary unit.
\end{theorem}

This theorem implies immediately the convergence for the spectrum of $\He$ \cite[Ch. V\!I\!I\!I, Sec. 7, Ths. V\!I\!I\!I.23, V\!I\!I\!I.24]{RS}.

\begin{theorem}\label{th2.2}
As $\e\to+0$, the spectrum of  operator $\He$ converges to that of $\Ho$. Namely, if $\lambda$ is not in the spectrum of $\Ho$, for sufficiently small $\e$ the same is true for $\He$. And if $\lambda$ is in the spectrum of $\Ho$, for each $\e$ there exists $\lambda_\e$ in the spectrum of $\He$ such that $\lambda_\e\to\lambda$ as $\e\to+0$.
\end{theorem}

Since operator $\He$ is periodic, it has a band spectrum described via Floquet-Bloch decomposition. Namely, let
\begin{equation*}
\g_\e^-:=\{x:\, x_1=0,\, |x_2|<\e\},\quad \g_\e^+:=\{x:\, x_1=d,\, |x_2|<\e\},\quad \G_\e:=\p\Pi\setminus(\g_\e^-\cup\g_\e^+).
\end{equation*}
In $\Pi$ we introduce the operator
\begin{equation*}
\He(\tau):=-\D + V
\end{equation*}
subject to the Dirichlet condition on $\G_\e$ and to quasiperiodic conditions $\g_\e^\pm$:
\begin{equation}\label{2.3a}
u\big|_{\g_\e^+}=e^{\iu\tau d} u\big|_{\g_\e^-},
\qquad \frac{\p u}{\p x_1}\Big|_{\g_\e^+}=e^{\iu\tau d} \frac{\p u}{\p x_1}\Big|_{\g_\e^-},
\end{equation}
where $\tau\in[-\frac{\pi}{d},\frac{\pi}{d})$ is the quasimomentum. As for operator $\He$, we define $\He(\tau)$ as the self-adjoint operator in $L_2(\Pi)$ associated with the sesquilinear form
\begin{equation*}
\he^\tau(u,v):=\big(\nabla u,\nabla v\big)_{L_2(\Pi)} + (Vu,v)_{L_2(\Pi)}
\end{equation*}
in $L_2(\Pi)$ on the domain
\begin{equation*}
\Dom(\he^\tau):=\big\{u\in \Wo^1(\Pi,\G_\e):\, \text{the first condition in (\ref{2.3a}) is satisfied}\}.
\end{equation*}

We introduce one more operator $\Ho^\Pi:=-\D+V$ in $\Pi$ subject to the Dirichlet condition; the associated sesquilinear form $\ho^\Pi$ is given by the same expression as for $\he^\tau$, but the domain is $\Wo^1(\Pi)$.

By $\spec(\cdot)$, $\essspec(\cdot)$, $\discspec(\cdot)$ we denote the spectrum, the essential spectrum, and the discrete spectrum of an operator.

The spectrum of $\He$ is described via that of $\He(\tau)$ \cite{Ku}:
\begin{equation}\label{2.5}
\spec(\He)=\bigcup\limits_{\tau\in\left[-\frac{\pi}{d},\frac{\pi}{d}\right)} \spec(\He(\tau)).
\end{equation}
Potentials $A$ and $V$ are compactly supported in $\overline{\Pi}$ and in accordance with \cite{Bi}:
\begin{equation}\label{2.6}
\essspec(\He(\tau))=\essspec(\Ho^\Pi)=[\pi^2,+\infty).
\end{equation}
Here we have also employed that domain $\Pi$ has two straight outlets at infinity.

The spectrum of $\Ho$ is purely essential and it reads as
\begin{equation}\label{2.7}
\spec(\Ho)=\essspec(\Ho)=\discspec(\Ho^\Pi)\cup [\pi^2,+\infty)
\end{equation}
and each discrete eigenvalue of $\Ho^\Pi$ is an eigenvalue for $\Ho$ of infinite multiplicity.

In view of (\ref{2.5}), (\ref{2.6}), the spectrum of $\He$ contains the band $[\pi^2,+\infty)$. It can also have additional bands generated by the isolated eigenvalues of $\He(\tau)$ located below the threshold of the essential spectrum. As $\e\to+0$, these eigenvalues converge to the discrete eigenvalues of $\Ho^\Pi$ or to $\pi^2$, see (\ref{2.7}) and Theorem~\ref{th2.2}.

The second part of our paper is devoted to the eigenvalues of $\He(\tau)$ converging to $\pi^2$ as $\e\to+0$. In other words, we study the bands in the spectrum $\spec(\He)$ close to $\pi^2$ for small $\e$. We shall see that the existence and structure of these bands are highly sensible to the existence of bounded solutions to the problem
\begin{equation}\label{2.8}
\big(-\D+V-\pi^2\big)\psi_0=0\quad \text{in}\quad\Pi,\qquad \psi_0=0\quad \text{on}\quad \p\Pi.
\end{equation}

Let $(r_-,\theta_-)$ be the polar coordinates associated with variables $(-x_2,x_1)$, and $(r_+,\theta_+)$ be the polar coordinates associated with $(x_2,d-x_1)$.

Denote
\begin{equation*}
\ell_\tau(\psi):=\frac{\p\psi}{\p x_1}(\g_0^-)+\frac{\p\psi}{\p x_1}(\g_0^+)e^{-\iu\tau d}, \quad \ell_\tau'(\psi):=\frac{\p^2\psi}{\p x_1\p x_2}(\g_0^-)+\frac{\p^2\psi}{\p x_1\p x_2}(\g_0^+)e^{-\iu\tau d},
\end{equation*}
where $\g_0^-:=(0,0)$, $\g_0^+:=(0,d)$.
Our second main result is as follows.

\begin{theorem}\label{th2.3} 1) Suppose that problem (\ref{2.8}) has no bounded nontrivial solutions. Then for sufficiently small $\e$ and some $\d>0$ the segment $[\pi-\d,\pi)$ contains no spectrum of $\He(\tau)$ for each $\tau\in\left[-\frac{\pi}{d},\frac{\pi}{d}\right)$.

2) Suppose that problem (\ref{2.8}) has one nontrivial bounded solution satisfying
\begin{gather}
\left|\frac{\p\psi_0}{\p x_1}(\g_0^-)\right|\not= \left|\frac{\p\psi}{\p x_1}(\g_0^+)\right|,\label{2.7a}
\\
\psi_0(x)=c_{0,\pm}\sin\pi(x_1-a_\pm) + O(e^{-\sqrt{3}\pi|x_2|}),\quad x_2\to\pm\infty,\label{2.9}
\end{gather}
where $c_{0,\pm}$ are some constants obeying $|c_{0,+}|^2+|c_{0,-}|^2=1$. Then for each $\tau \in\left[-\frac{\pi}{d},\frac{\pi}{d}\right)$, for all sufficiently small $\e$ in the vicinity of $\pi^2$ there exists exactly one eigenvalue $\l_\e(\tau)$ of $\He(\tau)$. This eigenvalue is simple and has the asymptotics
\begin{equation}\label{2.13}
\l_\e(\tau)=\pi^2-\pi^2|\ell_\tau(\psi_0)|^2\e^4+O(\e^6|\ln\e|),\quad \e\to+0,
\end{equation}
where the estimate for the error term is uniform in $\tau$.

3) Suppose that problem (\ref{2.8}) has two linearly independent bounded solutions $\psi_0^{(j)}$, $j=1,2$, satisfying
\begin{gather}
\left|\frac{\p\psi_0^{(1)}}{\p x_1}(\g_0^-)\right|\not= \left|\frac{\p\psi_0^{(1)}}{\p x_1}(\g_0^+)\right|%,\quad %L_\tau\nparallel L'_\tau,
\label{2.8a}
\\
\label{2.10}
\psi_0^{(j)}(x)= c_{0,\pm}^{(j)}\sin\pi(x_1-a_\pm) + O(e^{-\sqrt{3}\pi|x_2|}),\quad x_2\to\pm\infty,\quad j=1,2,
\end{gather}
where $c_{0,\pm}^{(j)}$ are some constants obeying
\begin{equation}\label{2.11}
|c_{0,+}^{(j)}|^2+|c_{0,-}^{(j)}|^2=1,\quad j=1,2,\quad c_{0,+}^{(1)}\overline{c_{0,+}^{(2)}}+ c_{0,-}^{(1)}\overline{c_{0,-}^{(2)}}=0,
\end{equation}
and
$L_\tau:=\big(\ell_\tau(\psi_0^{(1)}),\ell_\tau(\psi_0^{(2)})\big)$, $L'_\tau:=\big(\ell'_\tau(\psi_0^{(1)}),\ell'_\tau(\psi_0^{(2)})\big)$.

Then for each $\tau \in\left[-\frac{\pi}{d},\frac{\pi}{d}\right)$, for all sufficiently small $\e$ in the vicinity of $\pi^2$ there exist exactly two eigenvalues $\l_\e^{(j)}(\tau)$, $j=1,2$, of $\He(\tau)$. Both these eigenvalues are simple and have the asymptotics
\begin{align}\label{2.14}
&\l_\e^{(1)}(\tau)=\pi^2-\pi^2|L_\tau|^2\e^4+O(\e^6|\ln\e|),\quad\e\to+0,
\\
& \l_\e^{(2)}(\tau)=\pi^2-\mu(\tau)
\e^8\ln^2\e + O(\e^{10}|\ln\e|^2),\quad\e\to+0,\label{2.15}
\\
&\mu(\tau):=
\frac{\pi^2}{16} \frac{\big(\|L_\tau\|_{\mathds{C}^2}^2\|L'_\tau\|_{\mathds{C}^2}^2 - |(L_\tau,L'_\tau)_{\mathds{C}^2}|^2\big)^2}{\|L_\tau\|_{\mathds{C}^2}^4}, \quad\e\to+0,\nonumber%\l%abel{2.20}
\end{align}
where the estimates for the error terms are uniform in $\tau$.
\end{theorem}

The above theorem gives us a nice opportunity to describe the band spectrum of $\He$ in the vicinity of $\pi^2$. If problem (\ref{2.8}) has no nontrivial bounded solutions, operator $\He$ has no spectral bands below $\pi^2$. Namely, for some $\d>0$ in $[\pi^2-\d,\pi^2)$ there is no spectrum of $\He$ provided $\e$ is sufficiently small.

Once problem (\ref{2.8}) has one bounded nontrivial solution, eigenvalue $\l_\e(\tau)$ of operator $\He(\tau)$ generates the spectral band $S_\e:=[\min\limits_{\tau}\l_\e(\tau), \max\limits_{\tau}\l_\e(\tau)]$ for $\He(\tau)$. Asymptotics (\ref{2.13}) allows us to describe the behaviour of the end-points for this band:
%\begin{equation}\l%abel{2.16}
\begin{align*}
\min\limits_{\tau} \l_\e(\tau)=&\pi^2-\pi^2\e^4 \max\limits_{\tau} |\ell_\tau(\psi_0)|^2 + O(\e^6|\ln\e|)
\\
=& \pi^2-\pi^2\e^4 \left(\left|\frac{\p\psi_0}{\p x_1}(\g_0^-)\right|+\left|\frac{\p\psi_0}{\p x_1}(\g_0^+)\right|\right)^2
+ O(\e^6|\ln\e|)
\end{align*}
%\end{equation}
and
%\begin{equation}\l%abel{2.17}
\begin{align*}
\max\limits_{\tau} \l_\e(\tau)=&\pi^2-\pi^2\e^4 \min\limits_{\tau} |\ell_\tau(\psi_0)|^2 + O(\e^6|\ln\e|)
\\
=& \pi^2-\pi^2\e^4 \left(\left|\frac{\p\psi_0}{\p x_1}(\g_0^-)\right|-\left|\frac{\p\psi_0}{\p x_1}(\g_0^+)\right|\right)^2
+ O(\e^6|\ln\e|).
\end{align*}
%\end{equation}
In view of assumption (\ref{2.7a}), both the quantities $\left|\frac{\p\psi_0}{\p x_1}(\g_0^-)\right|\pm \left|\frac{\p\psi}{\p x_1}(\g_0^+)\right|$ are non-zero and hence, band $S_\e$ is separated from $\pi^2$. It means that the presence of the small windows coupling the cells in domain $\Om_\e$ turns the virtual level at $\pi^2$ into a small spectral band below $\pi^2$. We stress that this band disappear once $\e=0$.

A similar situation happens if problem (\ref{2.8})  has two bounded non-trivial solutions. Eigenvalues $\l_\e^{(j)}(\tau)$, $j=1,2$, generate two spectral bands $S_\e^{(j)}:=[\min\limits_{\tau}\l_\e^{(j)}(\tau), \max\limits_{\tau}\l_\e^{(j)}(\tau)]$. Asymptotics (\ref{2.14}) implies
\begin{align}\label{2.18}
&
\begin{aligned}
&
\min\limits_{\tau} \l_\e^{(1)}(\tau)=\pi^2-\pi^2\e^4 \big(\|p_1\|_{\mathds{C}^2}^2+\|p_2\|_{\mathds{C}^2}^2 + 2 |(p_1,p_2)_{\mathds{C}^2}|\big) + O(\e^6|\ln\e|),
\\
&
\max\limits_{\tau} \l_\e^{(1)}(\tau)=\pi^2-\pi^2\e^4 \big(\|p_1\|_{\mathds{C}^2}^2+\|p_2\|_{\mathds{C}^2}^2 - 2 |(p_1,p_2)_{\mathds{C}^2}|\big) + O(\e^6|\ln\e|),
\end{aligned}
\\
&p_1:=\left(\frac{\p\psi_0^{(1)}}{\p x_1}(\g_0^-),\frac{\p\psi_0^{(2)}}{\p x_1}(\g_0^-)\right),\quad p_2:=\left(\frac{\p\psi_0^{(1)}}{\p x_1}(\g_0^+),\frac{\p\psi_0^{(2)}}{\p x_1}(\g_0^+)\right).\nonumber
\end{align}
Band $S_\e^{(1)}$ is separated from $\pi^2$ thanks to assumption (\ref{2.8a}) and the above asymptotics.

It follows from (\ref{2.15}) that
%\begin{equation}\l%abel{2.19}
\begin{align*}
&\min\limits_{\tau} \l_\e^{(2)}(\tau)=\pi^2-\pi^2\e^8\ln^2\e\,\max\limits_{\tau} \mu(\tau) + O(\e^{10}|\ln\e|),
\\
&\max\limits_{\tau} \l_\e^{(2)}(\tau)=\pi^2-\pi^2\e^8\ln^2\e\,\min\limits_{\tau} \mu(\tau) + O(\e^{10}|\ln\e|).
\end{align*}
%\end{equation}
We first observe that bands $S_\e^{(1)}$ and $S_\e^{(2)}$ do not intersect since the distance from $S_\e^{(1)}$ to $\pi^2$ is of order $O(\e^4)$, see (\ref{2.18}), while similar distance for $S_\e^{(2)}$ is at most $O(\e^8\ln^2\e)$. If $\mu(\tau)>0$, band $S_\e^{(2)}$ is separated from $\pi^2$. In distinction to the above discussed cases, function $\mu(\tau)$ can attain its minimum or/and maximum at the internal points of Brillouin zone $\left[-\frac{\pi}{d},\frac{\pi}{d}\right)$ apart from  the center and the end-points. If it happens, band function $\l_\e^{(2)}(\tau)$ also attains its minimum or/and maximum in the internal points of the Brillouin zone. This is an interesting phenomenon discovered recently for elliptic operators with small periodic perturbations \cite{BP1}, \cite{BP2}, \cite{BP3}, \cite{Na2}, \cite{Na6}. The present paper provide one more example of a periodic operator with such property. The difference is that in our case the bands converge to the threshold of the essential spectrum of operator $\Ho^\Pi$, not to a discrete eigenvalue.

It should be also stressed that operator $\Ho^\Pi$ can have also discrete eigenvalues below the essential spectrum. While the windows open, these eigenvalues generate spectral bands for operator $\He$. Their structure and asymptotic behavior can be studied exactly in the same way as it was done in \cite{RJMP15} for a similar operator but with the Neumann condition.

\section{Norm resolvent convergence}

In this section we prove Theorem~\ref{th2.1}. We let
\begin{equation*}
B_\e^k:=\{x:\, |x-(kd,0)|<3\e\},\quad \Pi^k:=\{x:\, x-(kd,0)\in\Pi\}.
\end{equation*}
The first main ingredient of the proof is the following auxiliary lemma.

\begin{lemma}\label{lm3.1}
For sufficiently small $\e$, for each $u\in\Wo^2(\Om)$, $v\in\Wo^1(\Om_\e)$ the estimates
\begin{align}
&\sum\limits_{k\in\mathds{Z}} \|u\|_{L_2(B_\e^k)}^2 \leqslant C\e^4|\ln\e| \|u\|_{W_2^2(\Om)}^2,
\label{3.1}
\\
&\sum\limits_{k\in\mathds{Z}} \|\nabla u\|_{L_2(B_\e^k)}^2 \leqslant C\e^2|\ln\e| \|u\|_{W_2^2(\Om)}^2,
\label{3.5}
\\
&\sum\limits_{k\in\mathds{Z}} \|v\|_{L_2(B_\e^k)}^2 \leqslant C\e^2 \|v\|_{W_2^1(\Om)}^2
\label{3.2}
\end{align}
hold true, where $C$ is a constant independent of $\e$, $u$, and $v$.
\end{lemma}

\begin{proof}
We begin with proving (\ref{3.2}). In each $B_\e^k$ we introduce rescaled variables  $\xi^k:=(\xi^k_1,\xi^k_2)=(x_1-kd,x_2)\e^{-1}$. Function $\widetilde{v}(\xi^k):=v(\e\xi_1^k+kd,\e\xi_2^k)$ belongs to $\Wo^1(B_1,\G)$, where $B_1:=\{\xi^k: |\xi^k|<3\}$, $\G:=\{\xi:\, \xi_1^k:=0,\, 1<|\xi_1^k|<3\}$. And since $\widetilde{v}$ vanishes on $\G$, we have the estimate
\begin{equation}\label{3.3}
\|\widetilde{v}\|_{L_2(B_1)}^2\leqslant C\|\nabla_\xi \widetilde{v}\|_{L_2(B_1)}^2,
\end{equation}
where $C$ is a constant independent of $\widetilde{v}$. Rewriting this estimate in variables  $x$, we get
\begin{equation}\label{3.10}
\|v\|_{L_2(B_\e^k)}^2\leqslant C\e^2 \|\nabla v\|_{L_2(B_\e^k)}^2
\end{equation}
with the same constant $C$ as in (\ref{3.3}). Summing up this inequality over $k\in\mathds{Z}$, we arrive at (\ref{3.2}).

We proceed to inequalities (\ref{3.1}), (\ref{3.5}). Since $u$ is also an element of $\Wo^1(\Om_\e)$, it satisfies (\ref{3.10}). Then proceeding as in the proof of Lemma~3.2 in \cite{OSHY}, one can prove easily the estimate
\begin{equation*}
\|\nabla u\|_{L_2(B_\e^k)}^2\leqslant C\e^2|\ln\e|\|u\|_{W_2^2(B_\d^k)}^2
\end{equation*}
for some fixed $\d>0$, where constant $C$ is independent of $\e$, $k$, and $u$. The obtained estimate and (\ref{3.10}) with $v=u$ lead us to (\ref{3.5}), (\ref{3.2}). \end{proof}

Given $f\in L_2(\Pi)$, we let $u_\e:=(\He-\iu)^{-1}f$, $u_0:=(\Ho-\iu)^{-1}f$. By $\chi_1=\chi_1(t)$ we denote an infinitely differentiable cut-off function being one as $t<2$ and vanishing for $t>3$. We also denote $\chi_\e(x):=\sum\limits_{k\in\mathds{Z}} \chi_1(|x-(kd,0)|\e^{-1}) $.

By straightforward calculations one can check that the function
$v_\e:=u_\e-\chi_\e u_0$ solves the equation
\begin{equation*}
(\He-\iu)v_\e=\chi_\e f + g_\e,\quad g_\e:=-2\nabla\chi_\e\cdot \nabla u_0 - u_0\D\chi_\e.
\end{equation*}
We write the associated integral identity: %in terms of the quadratic %form:
\begin{equation}\label{3.4}
\he(v_\e,v_\e)-\iu\|v_\e\|_{L_2(\Om_\e)}^2=(\chi_\e f, v_\e)_{L_2(\Om_\e)} + (g_\e, v_\e)_{L_2(\Om_\e)}.
\end{equation}
Let us estimate the right hand side of this identity.

Function $\chi_\e$ is non-zero only in $B_k^\e$, and  $|\nabla \chi_\e|\leqslant C\e^{-1}$, $|\D\chi_\e|\leqslant C\e^{-2}$, where $C$ is a constant independent of $x$ and $\e$, and $v_\e\in \Wo^1(\Om_\e)$, $u_0\in \Wo^2(\Om)$. Then by Lemma~\ref{lm3.1}, we have the estimates
\begin{equation}\label{3.6}
\big|(\chi_\e f,v_\e)_{L_2(\Om_\e)}\big|\leqslant C\|f\|_{L_2(\Om)} \|v_\e\|_{L_2\big(\bigcup\limits_{k\in\mathds{Z}}B_\e^k\big)} \leqslant C\e\|v_\e\|_{W_2^1(\Om_\e)}^2,
\end{equation}
and
\begin{equation}\label{3.7}
\big|(g_\e,v_\e)_{L_2(\Om_\e)}\big| \leqslant C\e|\ln\e|^{\frac{1}{2}} \|u\|_{W_2^2(\Om)} \|v_\e\|_{W_2^1(\Om_\e)}.
\end{equation}
Hereinafter by $C$ we denote various inessential constants independent of $\e$, $u$, $v_\e$, $f$.

By standard smoothness improving theorems we obtain
\begin{equation}\label{3.8}
\|u\|_{W_2^2(\Om)}\leqslant C\|f\|_{L_2(\Om)}.
\end{equation}
It is also easy to make sure that
\begin{equation*}
\|v_\e\|_{W_2^1(\Om_\e)}^2 \leqslant C\big|\he(v_\e,v_\e)-\iu\|v_\e\|_{L_2(\Om_\e)}^2\big|.
\end{equation*}
This estimate, (\ref{3.4}), (\ref{3.6}), (\ref{3.7}), (\ref{3.8}) yield
\begin{equation}\label{3.9}
\|v_\e\|_{W_2^1(\Om_\e)}\leqslant C\e|\ln\e|^{\frac{1}{2}}\|f\|_{L_2(\Om)}.
\end{equation}
In the same way how (\ref{3.7}) was proven, we get
\begin{equation*}
\|\chi_\e u_0\|_{W_2^1(\Om_\e)} \leqslant C\e|\ln\e|^{\frac{1}{2}} \|f\|_{L_2(\Om)}.
\end{equation*}
This inequality and (\ref{3.9}) imply
\begin{equation*}
\|u_\e-u_0\|_{W_2^1(\Om_\e)}\leqslant C\e |\ln\e|^{\frac{1}{2}} \|f\|_{L_2(\Om)}
\end{equation*}
that completes the proof.

\newpage

\section{Spectral bands}

This section is devoted to the proof of Theorem~\ref{th2.3}. Our proof consists of three parts. In the first part we study the existence of the eigenvalues of $\He(\tau)$ converging to $\pi^2$ as $\e\to+0$. In the second part we construct formally the asymptotic expansions for these eigenvalues in the case they exist. The third part is devoted to the justification of the formal asymptotics.

\subsection{Existence}\label{Subsect4.1}

We adapt the approach employed in \cite[Sect. 3,4]{MSb06} for studying a similar issue. In what follows we shall demonstrate only the main milestones omitting all the details borrowed directly from \cite{MSb06}.

Denote $\Pi^0:=\Pi\cap\{x:\, |x_2|<x_2^\infty+1\}$, $\Pi^\pm:=\Pi\cap\{x:\, \pm x_2>x_2^\infty+1\}$. Given $f\in L_2(\Pi)$ supported in $\Pi^0$, we consider the boundary value problem
\begin{equation}\label{4.1}
\big(-\D+V-\pi^2+k^2\big)u=f\quad\text{in}\quad\Pi, \qquad u=0\quad\text{on}\quad \p\Pi,
\end{equation}
where $k$ is a small complex parameter. We shall seek the generalized solution to this problem behaving at infinity as
\begin{equation}\label{4.2}
u(x,k)=c_\pm(k)e^{-k|x_2|}\sin\pi(x_1-a_\pm) + O\big(e^{-\sqrt{3\pi^2+k^2}|x_2|}\big),\quad x_2\to\pm\infty,
\end{equation}
where $c_\pm(k)$ are some constants.

Let $g\in L_2(\Pi^0)$ be an arbitrary function which we continue by zero in $\Pi^\pm$. We consider two boundary value problems
\begin{equation}\label{4.3}
(-\D-\pi^2+k^2)v^\pm=g\quad\text{in}\quad\Pi^\pm,\qquad v^\pm=0\quad\text{on}\quad\p\Pi^\pm,
\end{equation}
and we solve them by the separation of variables:
\begin{align*}
&v^\pm(x,k)=\sum\limits_{j=1}^{\infty} \int\limits_{\Pi^\pm} G_j^\pm(x,t,k)g(t)\di t,\quad t=(t_1,t_2),
\\
&G_j^\pm(x,t,k):=\frac{1}{s_j(k)}\big(e^{-s_j(k)|x_1-t_1|}-e^{\mp s_j(k)(x_1+t_1)}\big)\sin\pi jx_2\sin\pi j t_2,
\end{align*}
where $s_1(k):=k$, $s_j(k):=\sqrt{\pi^2(j^2-1)+k^2}$, $j\geqslant 2$. As $k=0$, we let
\begin{equation*}
G_1^\pm(x,t,0):=\big(-|x_1-t_1|\pm(x_1+t_1)\big) \sin\pi x_2 \sin\pi t_2.
\end{equation*}
We introduce one more function: $v:=v^\pm$ in $\Pi^\pm$, $v=0$ in $\Pi^0$.

We consider the boundary value problem
\begin{equation}\label{4.4}
-\D w=g+(\pi^2-k^2)v\quad\text{in}\quad\Pi^0,\qquad w=v\quad\text{on}\quad \p\Pi^0.
\end{equation}
This problem is uniquely solvable.

Let $\chi_2=\chi_2(x_2)$ be an infinitely differentiable cut-off function equalling one as $|x_2|<x_2^\infty$ and vanishing for $|x_2|>x_2^\infty+1$. We construct the solution to problem (\ref{4.1}), (\ref{4.2}) as
\begin{equation}\label{4.5}
u(x,k)=\chi_2(x_2) w(x,k) + (1-\chi_2(x_2)) v(x,k).
\end{equation}
This function satisfies the required boundary condition and (\ref{4.2}). Substituting (\ref{4.5}) into the equation in (\ref{4.1}) and employing the equations in (\ref{4.3}), (\ref{4.4}), we arrive at the equation for $g$:
\begin{gather}\label{4.6}
(I+\A_1(k))g=f,
\\
\A_1(k)g:=(w-v)(-\D-\pi^2+k^2)\chi_2 - 2\nabla\chi_2\cdot \nabla(w-v).
%\\
%&+(2\iu A\cdot\nabla + |A|^2+\iu\mathrm{div})(\chi_2 w+ %(1-\chi_2)v).
%\end{aligned}
\nonumber
\end{gather}
Reproducing now arguments from \cite{MSb06}, one can prove  the following statements. Equation (\ref{4.6}) is equivalent to problem (\ref{4.1}), (\ref{4.2}). Operator $\A_1$ is compact in $L_2(\Pi^0)$ and is holomorphic w.r.t. $k$. If problem (\ref{2.8}) has no nontrivial bounded generalized solutions, the inverse operator $(I+\A_1(k))^{-1}$ is bounded for all sufficiently small $\e$. If problem (\ref{2.8}) has nontrivial bounded generalized solutions described in the formulation of Theorem~\ref{th2.3}, the inverse operator $(I+\A_1(k))^{-1}$ has a first order pole at $k=0$:
\begin{equation*}%\l%abel{4.53}
(I+\A_1(k))^{-1}=\frac{g_0}{k}(\cdot\,,\psi_0)_{L_2(\Pi^0)} + \A_2(k),
\end{equation*}
provided problem (\ref{2.8}) has one nontrivial solution satisfying (\ref{2.9}), and
\begin{equation*}%\l%abel{4.54}
(I+\A_1(k))^{-1}=\frac{1}{k}\sum\limits_{j=1}^{2}g_0^{(j)}(\cdot\,,\psi_0^{(j)})_{L_2(\Pi^0)} + \A_2(k),
\end{equation*}
provided problem (\ref{2.8}) has two nontrivial solutions satisfying (\ref{2.10}). Here $\A_2(k)$ is a bounded operator in $L_2(\Pi^0)$
holomorphic w.r.t. small complex $k$, and $g_0, g_0^{(j)}\in L_2(\Pi^0)$ are the functions generating the above mentioned solutions $\psi_0$, $\psi_0^{(j)}$ of problem (\ref{2.8}) by formula (\ref{4.5}) with $g=g_0$, $g=g_0^{(j)}$.

We proceed to the eigenvalue equation $\He(\tau)$, which we write as the boundary value problem
\begin{equation}\label{4.7}
\big(-\D+V\big)\psi_\e=(\pi^2-k^2)\psi_\e\quad\text{in}\quad\Pi, \qquad \psi_\e=0\quad\text{on}\quad \G_\e,
\end{equation}
with boundary conditions (\ref{2.3a}). At infinity, we again assume behavior (\ref{4.2}). Then the eigenvalues of $\He(\tau)$ correspond to values $k$ with $\RE k>0$, for which problem (\ref{4.7}), (\ref{4.2}) has a nontrivial solution.

Problem (\ref{4.7}), (\ref{4.2}) can be also reduced to an operator equation like (\ref{4.6}); one just should replace problem (\ref{4.4}) by
\begin{equation*}
-\D w_\e=g+(\pi^2-k^2)v\quad\text{in}\quad\Pi^0,\qquad w_\e=v\quad\text{on}\quad \p\Pi^0\setminus(\g_\e^-\cup\g_\e^+)
\end{equation*}
with boundary conditions (\ref{2.3a}). The corresponding equation in $L_2(\Pi^0)$ reads as
\begin{equation}\label{4.48}
(I+\A_3(k,\e,\tau))g=0,
\end{equation}
where $\A_3(k,\e,\tau)$ is a compact operator in $L_2(\Pi^0)$ holomorphic w.r.t. $k$ for each $\tau$ and $\e$. It is defined by the same expression as $\A_1(k)$.

Proceeding as in the proof of Theorem~\ref{th2.1}, it is easy to make sure that
\begin{equation*}
\|w_\e-w\|_{W_2^1(\Pi^0)}\leqslant C\e|\ln\e|^{\frac{1}{2}}\|g\|_{L_2(\Pi^0)},
\end{equation*}
where $C$ is a constant independent of $\e$, $\tau$, and $g$. This estimate and the definition of $\A_3$ yield the estimate
\begin{equation*}%\l%abel{4.49}
\|\A_3-\A_1\|_{L_2(\Pi^0\to L_2(\Pi^0))}\leqslant C\e|\ln\e|^{\frac{1}{2}},
\end{equation*}
where $C$ is a constant independent of $\e$ and $\tau$. Then arguing as in \cite{MSb06}, one can prove the following lemma.

\begin{lemma}\label{lm4.1}

\noindent 1) If problem (\ref{2.8}) has no nontrivial bounded solutions, for sufficiently small $\e$ and some $\d>0$ the segment $[\pi-\d,\pi)$ contains no spectrum of $\spec(\He)$.

\noindent 2) If problem (\ref{2.8}) has nontrivial bounded solutions, there exists $k=k_\e\to0$  for which problem (\ref{4.7}), (\ref{4.2}) has a nontrivial solution. The total amount of such linearly independent solutions associated with all $k_\e$ coincides with the number of nontrivial bounded solutions to problem (\ref{2.8}).
\end{lemma}

As we see, the first part of Theorem~\ref{th2.3} is proven and it remains to study the case when problem (\ref{2.8}) has nontrivial bounded generalized solution(s). In the next subsection we construct the asymptotic expansions for aforementioned values of $k_\e$ and we shall show that the associated nontrivial solution to (\ref{4.7}), (\ref{4.2}) is an eigenfunction of $\He(\tau)$.

\subsection{Asymptotics}

In this subsection we construct asymptotic expansions for values $k_\e$, for which problem (\ref{4.7}), (\ref{4.2}) has a nontrivial solution. Once we know the asymptotics, we shall check whether $\RE k_\e>0$, and if so, in accordance with (\ref{4.2}), the associated nontrivial solution to (\ref{4.7}) is an eigenfunction and $\pi^2-k_\e^2$ is an eigenvalue of $\He(\tau)$.

First we construct the asymptotics formally and then we shall discuss the justification, i.e., how to estimate the error terms. The formal construction is based on the combination of the method of matching asymptotic expansions \cite{Il} and the approach suggested in \cite{PMA15}.

We shall demonstrate the formal construction for the case when problem (\ref{2.8}) has two nontrivial solutions satisfying (\ref{2.10}). The case of one nontrivial solution is simpler and can be treated in the same way.

In accordance with Lemma~\ref{lm4.1}, there are two linearly independent nontrivial solutions $\psi_\e^{(j)}$, $j=1,2$,  to problem (\ref{4.7}), (\ref{4.2}) associated with one or two values $k_\e^{(j)}\to0$, $j=1,2$. By analogy with \cite[Lm. 3.1]{RJMP15}, one can make sure that there exists a unitary $2\times2$ matrix with entries $a_{ij}=a_{ij}(\tau)$ such that the functions $\Psi_0^{(i)}(x,\tau)=a_{i1}(\tau)\psi_0^{(1)}(x)+a_{i2}(\tau)\psi_0^{(2)}(x)$ satisfy (\ref{2.10}), (\ref{2.11}) and
\begin{equation}\label{2.12}
\ell_\tau(\Psi_0^{(1)})\not=0,\quad \ell_\tau(\Psi_0^{(2)})=0\quad\text{for each}\quad \tau\in \left[-\frac{\pi}{d},\frac{\pi}{d}\right).
\end{equation}

 We assume the following ans\"atzes for $k_\e^{(j)}$ and $\psi_\e^{(j)}$:
\begin{align}
&k_\e^{(j)}(\tau)=\e^2 k_2^{(j)}(\tau)+\e^4( k_4^{(j)}(\tau)+k_{4,1}^{(j)}(\tau)\ln\e)+\ldots,\label{4.8a}
\\
&\psi_\e^{(j)}(x,\tau)=\Psi_0^{(j)}(x)+\e^2 \Psi_2^{(j)}(x,\tau)+\e^4\big(\Psi_4^{(j)}(x,\tau) +\Psi_{4,1}^{(j)}(x,\tau)\ln\e\big)+\ldots\label{4.8b}
\end{align}
The latter ans\"atz will be used outside a small neighbourhood of $\g_\e^\pm$, and in what follows, this ans\"atz will be referred to as external expansion. Hereinafter by ``\ldots'' we indicate lower order terms.

We substitute (\ref{4.8a}), (\ref{4.8b}) into (\ref{4.7}), equate the coefficients at the like powers of $\e$, and replace $\g_\e^\pm$ by $\g_0^\pm$. It leads us to the boundary value problems for $\Psi_2^{(j)}$, $\Psi_4^{(j)}$, $\Psi_{4,1}^{(j)}$:
\begin{gather}
 \big(-\D+V-\pi^2\big)\Psi_2^{(j)}=0 \quad\text{in}\quad\Pi,\qquad\Psi_2^{(j)}=0 \quad\text{on}\quad \p\Pi\setminus(\g_0^+\cup\g_0^-),\label{4.9}
\\
\big(-\D+V-\pi^2\big)\Psi_{4,1}^{(j)}=0 \quad\text{in}\quad\Pi,\qquad\Psi_{4,1}^{(j)}=0 \quad\text{on}\quad \p\Pi\setminus(\g_0^+\cup\g_0^-),\label{4.9a}
\\
\begin{gathered}
  \big(-\D+V-\pi^2\big)\Psi_4^{(j)}+(k_2^{(j)})^2\Psi_0^{(j)} =0 \quad\text{in}\quad\Pi,
  \\
   \Psi_4^{(j)}=0 \quad\text{on}\quad \p\Pi\setminus(\g_0^+\cup\g_0^-),
\end{gathered}\label{4.10}
\end{gather}

Let us describe the behavior of functions $\Psi_2^{(j)}$, $\Psi_4^{(j)}$, $\Psi_{4,1}^{(j)}$ at infinity. In order to do it, we employ the approach suggested in \cite{PMA15}.

In $\Pi^\pm$, equation (\ref{4.7}) becomes equation (\ref{4.3}) with $g=0$ and thus,
\begin{equation}\label{4.13}
\Psi_\e^{(j)}(x,\tau)=\sum\limits_{p=1}^{\infty} b_p^{(j,\pm)}(\e,\tau) e^{\mp s_p(k_\e^{(j)})x_2}\sin\pi p (x_2-a_\pm),
\end{equation}
where $b_p^{(j,\pm)}(\e,\tau)$ are some constants. Similar to (\ref{4.8a}), we assume that
\begin{equation}\label{4.14}
b_p^{j,\pm}(\e,\tau)=b_{p,0}^{(j,\pm)}(\tau)+\e^2 b_{p,2}^{(j,\pm)}(\tau) + \e^4(b_{p,4}^{(j,\pm)}(\tau)+b_{p,4,1}^{(j,\pm)}(\tau)\ln\e) + \ldots,
\end{equation}
where $b_{p,0}^{j,\pm}$ are the coefficients for $\Psi_0^{(j)}$ in the expansion (\ref{4.13}) with $k_\e=0$, $\e=0$. In particular, $b_{1,0}^{j,\pm}=c_{0,\pm}^{(j)}$, see (\ref{2.10}).

We substitute (\ref{4.14}) and (\ref{4.8a}), (\ref{4.8b}) into (\ref{4.13}) and expand the right hand side into asymptotic series as $\e\to+0$ for each fixed $x$. Then we compare the coefficients at the like powers of $\e$ and $\ln\e$ in both sides of the obtained identity:
\begin{align}
\label{4.15}
&
\begin{aligned}
\Psi_2^{(j)}(x,\tau)= &(b_{1,2}^{(j,\pm)}(\tau)\mp b_{1,0}^{(j,\pm)}(\tau)  k_2^{(j)}(\tau)x_2) \sin\pi (x_1-a_\pm)
\\
&+\sum\limits_{p=2}^{\infty}e^{\mp s_p(0)x_2}b_{p,2}^{(j,\pm)}(\tau)\sin\pi p(x_1-a_\pm),
\end{aligned}
\\
&
 \begin{aligned}
 \Psi_{4,1}^{(j)}(x,\tau)= &(b_{1,4,1}^{(j,\pm)}(\tau) \mp b_{1,0}^{(j,\pm)}(\tau) k_{4,1}^{(j)}(\tau)x_2) \sin\pi (x_1-a_\pm)
 \\
 & +\sum\limits_{p=2}^{\infty}e^{\mp s_p(0)x_2}b_{p,4,1}^{(j,\pm)}(\tau)\sin\pi p(x_1-a_\pm),
 \end{aligned}\label{4.16}
\\
&
\begin{aligned}
\Psi_4^{(j)}(x,\tau)=& \bigg(b_{1,4}^{(j,\pm)}(\tau)\mp x_2(k_2^{(j)}b_{1,2}^{(j,\pm)}(\tau)+k_4^{(j)}(\tau)b_{1,0}^{(j,\pm)}(\tau))
\\
&+ \frac{1}{2} (k_2^{(j)})^2 x_2^2 b_{1,0}^{(j,\pm)}(\tau)\bigg)  \sin\pi (x_1-a_\pm)
\\
&+\sum\limits_{p=2}^{\infty}e^{\mp s_p(0)x_2}\left(b_{p,4}^{(j,\pm)}(\tau)\mp\frac{1}{2 s_p(0)} k_2^{(j)}(\tau) x_2 b_{p,0}^{(j,\pm)}(\tau) \right)\sin\pi p(x_1-a_\pm),
\end{aligned}\label{4.22}
\end{align}
as $\pm x_2>x_2^\infty$. The above formulae describe the desired behavior at infinity for the coefficients of the external expansion.

At points $\g_0^\pm$,  functions $\Psi_2^{(j)}$, $\Psi_4^{(j)}$, $\Psi_{4,1}^{(j)}$ are  to have singularities. The reason is that in the vicinity of $\g_\e^\pm$ the asymptotics for $\Psi_\e^{(j)}$ is constructed as an internal expansion. And in accordance with the method of matching asymptotic expansions, the external and internal expansions are to be matched \cite{Il}. It generates the singularities for functions $\Psi_2^{(j)}$, $\Psi_4^{(j)}$, $\Psi_{4,1}^{(j)}$, which we apriori introduce right now:
\begin{align}
&
 \begin{aligned}
\Psi_2^{(j)}(x,\tau)=&\a_{1,\pm}^{(j,2)}(\tau) r_\pm^{-1} \sin \theta_\pm  +
\a_{2,\pm}^{(j,2)}(\tau) r_\pm\sin\theta_\pm
\\
&+ \frac{\a_{1,\pm}^{(j,2)(\tau)}(V(0)-\pi^2)}{2} r_\pm \ln r_\pm\sin\theta_\pm  + \ldots,
 \end{aligned}\label{4.11}
\\
&\Psi_{4,1}^{(j)}(x,\tau)=\a_{1,\pm}^{(j,4,1)}(\tau) r_\pm^{-1} \sin \theta_\pm  + \ldots,\label{4.12a}
\\
&
\Psi_4^{(j)}(x,\tau)=\a_{1,\pm}^{(j,4)}(\tau) r_\pm^{-3} \sin 3\theta_\pm +  \a_{2,\pm}^{(j,4)}(\tau) r_\pm^{-2} \sin 2\theta_\pm +   \a_{3,\pm}^{(j,4)}(\tau) r_\pm^{-1} \sin \theta_\pm \ldots, % нужно уточнить
 \label{4.12b}
\end{align}
as $r_\pm\to0$, where $\a_{p,\pm}^{(j,q)}$, $\a_{1,\pm}^{(j,4,1)}$ are to be determined. In order to do it, we first observe that functions $\Psi_0^{(j)}$ satisfy an asymptotics similar to (\ref{4.11}):
%\begin{equation}\l%abel{4.17}
\begin{align*}
\Psi_0^{(j)}(x,\tau)=&\a_{1,\pm}^{(j,0)}(\tau)r_\pm \sin\theta_\pm + \a_{2,\pm}^{(j,0)}(\tau) r_\pm^2\sin 2\theta_\pm + \a_{3,\pm}^{(j,0)}(\tau) r_\pm^3\sin 3\theta_\pm
\\
&+ \a_{4,\pm}^{(j,0)}(\tau) r^3 \sin\theta_\pm \cos^2\theta_\pm +\ldots,\quad r_\pm\to0,
\end{align*}
%\end{equation}
where
\begin{equation}\label{4.21}
\begin{aligned}
&\a_{1,\pm}^{(j,0)}(\tau):=\mp\frac{\p\Psi_0^{(j)}}{\p x_1} (\g_0^\pm,\tau),\quad \a_{2,\pm}^{(j,0)}(\tau):=-\frac{\p^2\Psi_0^{(j)}}{\p x_1\p x_2}(\g_0^\pm,\tau),
\\
&\a_{3,\pm}^{(j,0)}(\tau):=\mp \frac{\p^3\Psi_0}{\p x_1^3}(\g_0^\pm,\tau),\quad
\a_{4,\pm}^{(j,0)}(\tau):=  \frac{V(0)-\pi^2}{2}\a_{1,\pm}^{(j,0)}(\tau).
\end{aligned}
\end{equation}
The above formula is just the Taylor expansion for $\Psi_0^{(j)}$ at $\g_0^\pm$ and some of the coefficients are determined by equation (\ref{2.8}) for $\Psi_0^{(j)}$.

We introduce the rescaled variables $\z^\pm:=(\z^\pm_1,\z^\pm_2)$, $\z^-:=(-x_2\e^{-1},x_1\e^{-1})$, $\z^+:=(x_2\e^{-1},(d-x_1)\e^{-1})$. Then in ans\"atz (\ref{4.8b}),
we replace $\Psi_2^{(j)}$, $\Psi_4^{(j)}$, $\Psi_{4,1}^{(j)}$ by their asymptotics (\ref{4.11}), (\ref{4.12a}), (\ref{4.12b}) and rewrite the result in variables $\z^\pm$:
\begin{equation*}%\l%abel{4.20}
\Psi_\e^{(j)}(x,\tau)=\e\Phi_{1,\pm}^{(j)}(\z^\pm,\tau) +\e^2\Phi_{2,\pm}^{(j)}(\z^\pm,\tau) + \e^3\ln\e \Phi_{2,1,\pm}^{(j)}(\z^\pm,\tau)+\e^3\Phi_{3,\pm}^{(j)}(\z^\pm,\tau)+\ldots
\end{equation*}
where
\begin{equation}\label{4.25}
\begin{aligned}
&\Phi_{1,\pm}^{(j)}(\z^\pm,\tau):=\a_{1,\pm}^{(j,0)}(\tau) \rho_\pm\sin\theta_\pm + \a_{1,\pm}^{(j,2)}(\tau)\rho_\pm^{-1}\sin\theta_\pm + \a_{1,\pm}^{(j,4)}(\tau)\rho_\pm^{-3}\sin3\theta_\pm,
\\
&\Phi_{2,\pm}^{(j)}(\z^\pm,\tau):= \a_{2,\pm}^{(j,0)}(\tau) \rho_\pm^2\sin 2\theta_\pm,\quad \Phi_{2,1,\pm}^{(j)}(\z^\pm,\tau):=\frac{\a_{1,\pm}^{(j,2)}(\tau)\big(V(0)-\pi^2\big)}{2}\rho_\pm \sin\theta_\pm,
\\
&\Phi_{3,\pm}^{(j)}(\z^\pm,\tau):=\a_{3,\pm}^{(j,0)}(\tau)\rho_\pm^3\sin 3\theta_\pm + \a_{4,\pm}^{(j,0)}(\tau)\rho_\pm^3\sin^2\theta_\pm \cos\theta_\pm  + \a_{2,\pm}^{(j,2)}(\tau)\rho_\pm \sin\theta_\pm
 \\
&\hphantom{\Phi_{3,\pm}^{(j)}(\z^\pm,\tau):=} + \frac{\a_{1,\pm}^{(j,2)}(\tau)(V(0)-\pi^2)}{2}\rho_\pm \ln\rho_\pm\sin\theta_\pm+\a_{3,\pm}^{(j,4)}(\tau)\rho_\pm^{-1}\sin\theta_\pm,
\end{aligned}
\end{equation}
where $\rho_\pm:=|\z^\pm|$. In accordance with the method of matching asymptotic expansions, the above identities yield that the internal expansions for $\psi_\e^{(j)}$ in the vicinities of $\g_\e^\pm$ should read as
\begin{equation}\label{4.26}
\psi_\e^{(j)}(x,\tau)=\e\phi_{1,\pm}^{(j)}(\z^\pm,\tau) + \e^2\phi_{2,\pm}^{(j)}(\z^\pm,\tau) + \e^2\ln\e\,\phi_{2,1,\pm}^{(j)}(\z^\pm,\tau)+ \e^3\phi_{3,\pm}^{(j)}(\z^\pm,\tau) +\ldots
\end{equation}
and its coefficients should satisfy the asymptotics
\begin{equation}\label{4.27}
\phi_{\natural,\pm}^{(j)}(\z,\tau)=\Phi_{\natural,\pm}^{(j)}(\z,\tau) + \ldots,\quad \z\to\infty,\quad \z_2>0,\qquad \z=(\z_1,\z_2).
\end{equation}

We substitute (\ref{4.26}), (\ref{4.8a}), (\ref{4.8b}) into (\ref{4.7}), (\ref{2.3a}), pass to variables $\z_\pm$, and equate the coefficients at the like powers of $\e$ and $\ln\e$ that gives the boundary value problems
\begin{equation}\label{4.28}
\begin{aligned}
&\D_\z \phi_{\natural,\pm}^{(j)}=0,\quad \z_2>0,\qquad \phi_{\natural,\pm}^{(j)}\big|_{\G_1}=0,
\\
&\phi_{\natural,+}^{(j)}\big|_{\g_1}= e^{\iu \tau d}\phi_{\natural,-}^{(j)}\big|_{\g_1},\qquad
\frac{\p\phi_{\natural,+}^{(j)}}{\p\z_2}\bigg|_{\g_1}= -e^{\iu \tau d}\frac{\p\phi_{\natural,-}^{(j)}}{\p\z_2}\bigg|_{\g_1},
\\
&\g_1:=\{\z:\, |\z_1^\pm|<1,\, \z_2^\pm=0\},\quad \G_1:=O\z_1\setminus\overline{\g_1},
\end{aligned}
\end{equation}
for $\natural=1$, $\natural=2$, $\natural=2,1$, and the problem
%\begin{equation}\l%abel{4.29}
\begin{align*}
&\D_\z \phi_{3,\pm}^{(j)}=(V(0)-\pi^2)\phi_{1,\pm}^{(j)},\quad \z_2>0,\qquad \phi_{3,\pm}^{(j)}\big|_{\G_1}=0,
\\
&\phi_{3,+}^{(j)}\big|_{\g_1}= e^{\iu \tau d}\phi_{3,-}^{(j)}\big|_{\g_1},\qquad
\frac{\p\phi_{3,+}^{(j)}}{\p\z_2}\bigg|_{\g_1}= -e^{\iu \tau d}\frac{\p\phi_{3,-}^{(j)}}{\p\z_2}\bigg|_{\g_1}.
\end{align*}
%\end{equation}

Equation in (\ref{4.28}) for $\phi_{1,\pm}^{(j)}$ has an explicit solution satisfying the boundary condition on $\G_1$:
\begin{equation}\label{4.30}
\phi_{1,\pm}^{(j)}(\z,\tau)=\b_{1,\pm}^{(j,1)}(\tau) \z_2 + \b_{1,\pm}^{(j,2)}(\tau) \IM \sqrt{z^2-1},\quad z=\z_1+\iu\z_2,
\end{equation}
where $\b_{1,\pm}^{(j,p)}$ are some constants and the branch of the square root is fixed by the restriction $\sqrt{1}=1$. To satisfy the desired boundary conditions on $\g_1$, the above constants should satisfy the identities:
\begin{equation}\label{4.31}
\b_{1,+}^{(j,2)}=e^{\iu\tau d} \b_{1,-}^{(j,2)},\quad
\b_{1,-}^{(j,1)}=-e^{\iu\tau d} \b_{1,-}^{(j,1)}.
\end{equation}
Extra two identities appear by comparing (\ref{4.30}) and asymptotics (\ref{4.27}), (\ref{4.25}):
\begin{equation}\label{4.32}
\a_{1,\pm}^{(j,0)}=\b_{1,\pm}^{(j,1)}+\b_{1,\pm}^{(j,2)}.
\end{equation}
These identities and (\ref{4.31}) yield the formulae for  constants $\b_{1,\pm}^{(j,p)}$:
\begin{equation}\label{4.33}
\begin{aligned}
&\b_{1,-}^{(j,1)}=\frac{\a_{1,-}^{(j,0)}-\a_{1,+}^{(j,0)}e^{-\iu\tau d}}{2},\qquad \b_{1,+}^{(j,1)}=\frac{\a_{1,+}^{(j,0)}-\a_{1,-}^{(j,0)}e^{\iu\tau d}}{2},
\\
&\b_{1,-}^{(j,2)}=\frac{\a_{1,-}^{(j,0)}+\a_{1,+}^{(j,0)}e^{-\iu\tau d}}{2},\qquad \b_{1,+}^{(j,2)}=\frac{\a_{1,+}^{(j,0)}+\a_{1,-}^{(j,0)}e^{\iu\tau d}}{2}.
\end{aligned}
\end{equation}
Formulae (\ref{4.30}) and (\ref{4.27}) determine also $\a_{1,\pm}^{(j,2)}$ and $\a_{1,\pm}^{(j,4)}$:
\begin{equation}\label{4.34}
\a_{1,\pm}^{(j,2)}=\frac{\b_{1,\pm}^{(j,2)}}{2},\quad \a_{1,\pm}^{(j,4)}=\frac{3 \b_{1,\pm}^{(j,2)}}{8}.
\end{equation}
Function $\phi_{2,1,\pm}^{(j)}$ is found exactly in the same way:
\begin{equation*}%\l%abel{4.35}
\phi_{2,1\pm}^{(j)}(\z,\tau)=\frac{(V(0)-\pi^2)}{4} \big(\b_{2,1,\pm}^{(j,1)}(\tau)\z_2+ \b_{2,1,\pm}^{(j,2)}(\tau)\IM\sqrt{z^2-1}\big),
\end{equation*}
where
%\begin{equation}\l%abel{4.36}
\begin{align*}
&\b_{2,1,-}^{(j,1)}:=\a_{1,-}^{(j,2)}-\a_{1,+}^{(j,2)}e^{-\iu\tau d}, &&  \b_{2,1,+}^{(j,1)}:= \a_{1,+}^{(j,2)}-\a_{1,-}^{(j,2)}e^{\iu\tau d},
\\
&\b_{2,1,-}^{(j,2)}=\a_{1,-}^{(j,2)}+\a_{1,+}^{(j,2)}e^{-\iu\tau d}, &&
\b_{2,1,+}^{(j,2)}:=\a_{1,+}^{(j,2)}+\a_{1,-}^{(j,2)}e^{\iu\tau d}.
\end{align*}
%\end{equation}
Function $\phi_{2,\pm}^{(j)}$ can also be written explicitly, namely,
\begin{equation*}%\l%abel{4.37}
\phi_{2,\pm}^{(j)}(\z,\tau)=2\a_{2,\pm}^{(j,1)}(\tau)\z_2^2+ \b_{2,\pm}^{(j,2)}(\tau)\z_2\IM z\sqrt{z^2-1},
\end{equation*}
where
\begin{equation}\label{4.35a}
\begin{aligned}
&\b_{2,-}^{(j,1)}=\frac{\a_{2,-}^{(j,0)}-\a_{2,+}^{(j,0)}e^{-\iu\tau d}}{2},\qquad \b_{2,+}^{(j,1)}=\frac{\a_{2,+}^{(j,0)}-\a_{2,-}^{(j,0)}e^{\iu\tau d}}{2},
\\
&\b_{2,-}^{(j,2)}=\frac{\a_{2,-}^{(j,0)}+\a_{2,+}^{(j,0)}e^{-\iu\tau d}}{2},\qquad \b_{2,+}^{(j,2)}=\frac{\a_{2,+}^{(j,0)}+\a_{2,-}^{(j,0)}e^{\iu\tau d}}{2}.
\end{aligned}
\end{equation}
In particular, it determines $\a_{2,\pm}^{(j,4)}$:
\begin{equation}\label{4.35b}
\a_{2,\pm}^{(j,4)}=\frac{\b_{2,\pm}^{(j,2)}}{8}.
\end{equation}

And finally, we are able to find explicitly $\phi_{3,\pm}^{(j)}$:
%\begin{equation}\l%abel{4.38}
\begin{align*}
\phi_{3,\pm}^{(j)}(\z,\tau)=& \b_{3,\pm}^{(j,1)}(\tau)\IM z^3+ \b_{3,\pm}^{(j,2)}(\tau) \IM\left((\sqrt{z^2-1})^3+\frac{3}{2}z\right)
\\
& +
\frac{V(0)-\pi^2}{2} \bigg( \b_{1,\pm}^{(j,1)}(\tau) \z_1^2\z_2+\frac{\b_{1,\pm}^{(j,2)}(\tau)}{2}\IM z^3
\\
&\hphantom{+
\frac{V(0)-\pi^2}{2} \bigg(}- \frac{\b_{1,\pm}^{(j,2)}(\tau)}{2}\z_2 \RE \Big(z\sqrt{z^2-1}-\ln\frac{z+\sqrt{z^2-1}}{2}-\frac{1}{2}\Big) \bigg)
\\
&
+ \b_{3,\pm}^{(j,3)}(\tau) \z_2 + \b_{3,\pm}^{(j,4)}(\tau) \IM \sqrt{z^2-1},
\end{align*}
%\end{equation}
%\begin{equation}\l%abel{4.38}
%\begin{aligned}
%\phi_{3,\pm}^{(j)}(\z,\tau)=& \frac{V(0)-\pi^2}{2} \left( %\b_{1,\pm}^{(j,1)}(\tau) \z_1^2\z_2 - %\frac{\b_{1,\pm}^{(j,2)}(\tau)}{2}\z_2 \RE %\Big(z\sqrt{z^2-1}-\ln\frac{z+\sqrt{z^2-1}}{2}-\frac{1}{2}\Big) %+\frac{\b_{1,\pm}^{(j,2)}(\tau)}{2}\IM z^3\right)
%\\
%&+\b_{3,\pm}^{(j,1)}(\tau)\IM z^3 + \b_{3,\pm}^{(j,2)}(\tau) %\IM\left((\sqrt{z^2-1})^3+\frac{3}{2}z\right)
%+ \b_{3,\pm}^{(j,3)}(\tau) \z_2 + \b_{3,\pm}^{(j,4)}(\tau) \IM %\sqrt{z^2-1},
%\end{aligned}
%\end{equation}
where $\b_{3,\pm}^{(j,p)}$ are some constants. And as for function $\phi_{1,\pm}^{(j)}$, coefficients $\b_{3,\pm}^{(j,p)}$ are determined by the boundary conditions on $\g_1$
and asymptotics (\ref{4.27}), (\ref{4.28}):
\begin{equation}\label{4.37a}
\begin{aligned}
&\b_{3,-}^{(j,1)}=\frac{\a_{3,-}^{(j,0)}-\a_{3,+}^{(j,0)}e^{-\iu\tau d}}{2},\qquad \b_{3,+}^{(j,1)}=\frac{\a_{3,+}^{(j,0)}-\a_{3,-}^{(j,0)}e^{\iu\tau d}}{2},
\\
&\b_{3,-}^{(j,2)}=\frac{\a_{3,-}^{(j,0)}+\a_{3,+}^{(j,0)}e^{-\iu\tau d}}{2},\qquad \b_{3,+}^{(j,2)}=\frac{\a_{3,+}^{(j,0)}+\a_{3,-}^{(j,0)}e^{\iu\tau d}}{2},
\\
&\b_{3,-}^{(j,3)}=\frac{\a_{2,-}^{(j,2)}-\a_{2,+}^{(j,2)}e^{-\iu\tau d}}{2},\qquad \b_{3,+}^{(j,3)}=\frac{\a_{2,+}^{(j,2)}-\a_{2,-}^{(j,2)}e^{\iu\tau d}}{2},
\\
&\b_{3,-}^{(j,4)}=\frac{\a_{2,-}^{(j,2)}+\a_{2,+}^{(j,2)}e^{-\iu\tau d}}{2},\qquad \b_{3,+}^{(j,4)}=\frac{\a_{2,+}^{(j,2)}+\a_{2,-}^{(j,2)}e^{\iu\tau d}}{2}.
\end{aligned}
\end{equation}
It yields
\begin{equation}\label{4.38a}
\a_{3,\pm}^{(j,4)}=\frac{4\b_{3,\pm}^{(j,4)}-3\b_{3,\pm}^{(j,2)}}{8}.
\end{equation}

Thus, for the coefficients of the external expansion we have problems (\ref{4.9}), (\ref{4.9a}), (\ref{4.10}), (\ref{4.15}), (\ref{4.16}), (\ref{4.22}), (\ref{4.11}), (\ref{4.12a}), (\ref{4.12b}). To study the solvability of these problems, we shall make use of the following auxiliary lemma.

\begin{lemma}\label{lm4.3}
Let a function $f\in L_2(\Pi)$ be compactly supported in $\overline{\Pi}$. Then the problem
\begin{equation}\label{4.39}
(-\D+V-\pi^2)u=f\quad\text{in}\quad\Pi,\qquad u=0\quad\text{on}\quad\p\Pi,
\end{equation}
has a generalized solution bounded at infinity if and only if
\begin{equation}\label{4.40}
\int\limits_{\Pi} f\overline{\Psi_0}=0.
\end{equation}
This solution is unique up to an additive term $C_1\Psi_0^{(1)}+C_2\Psi_0^{(2)}$, $C_1$, $C_2$ are arbitrary constants.
\end{lemma}

The proof of this lemma is based on rewriting problem (\ref{4.39}) to equation (\ref{4.6}) and proceeding then as in the proof of Lemma~4.7 in \cite{JMP11}.

Let us study the solvability of the problem of $\Psi_2^{(j)}$. By $\chi_{3,\pm}=\chi_{3,\pm}(x)$ we denote infinitely differentiable cut-off functions equalling one in a small neighborhood of points $\g_0^\pm$ and vanishing outside a bigger neighborhood. By $\chi_{4,\pm}=\chi_{4,\pm}(x_2)$ we denote infinitely differentiable cut-off functions equalling one for $\pm x_2>x_2^\infty+1$ and vanishing as $\pm x_2<x_2^\infty$.

We construct function $\Psi_2^{(j)}$  as $\Psi_2^{(j)}=\widetilde{\Psi}_2^{(j)}+\widehat{\Psi}_2^{(j)}$,
\begin{align*}
\widetilde{\Psi}_2^{(j)}(x,\tau)=&\a_{1,-}^{(j,2)}(\tau) \chi_{3,-}(x) r_-^{-1}\sin\theta_- + \a_{1,+}^{(j,2)}(\tau) \chi_{3,+}(x) r_+^{-1}\sin\theta_+
\\
&+ \big(b_{1,0}^{(j,-)}(\tau)\chi_{4,-}(x_2) - b_{1,0}^{(j,+)}(\tau)\chi_{4,+}(x_2)\big) x_2\sin\pi(x_1-a_\pm).
\end{align*}
Then for $\widehat{\Psi}_2^{(j)}$ we obtain problem (\ref{4.39}) with $f^{(j)}=(\D-V+\pi^2)\widetilde{\Psi}_2^{(j)}$. We write solvability condition (\ref{4.40}) for such $f$ as
\begin{equation*}
\lim\limits_{\genfrac{}{}{0 pt}{}{R\to+\infty}{\d\to+0}} \int\limits_{\Pi_{R,\d}} f\overline{\Psi}_0^{(p)}\di x=0,\quad \Pi_{R,\d}
:=\{x:\, |x_2|<R,\, |x-\g_0^\pm|>\d\},\quad p=1,2.
\end{equation*}
Integrating twice by parts in the above integrals and passing to the limit as $R\to+\infty$, $\d\to+0$, we arrive at the formulae
\begin{equation*}%\l%abel{4.41}
k_2^{(j)}\d_{jp}=2\pi\left(\a_{1,-}^{(j,2)}\overline{\a_{1,-}^{(p,0)}} + \a_{1,+}^{(j,2)}\overline{\a_{1,+}^{(p,0)}} \right).
\end{equation*}
Here we have also employed normalization conditions (\ref{2.11}). Substituting identities (\ref{4.32}), (\ref{4.33}), (\ref{4.21}), we finally obtain formulae for $k_2^{(j)}$:
\begin{equation}\label{4.42}
\begin{aligned}
k_2^{(j)}(\tau)=&2\pi (\b_{1,-}^{(j,2)}(\tau)\overline{\a_{1,-}^{(j,0)}(\tau)} + \b_{1,-}^{(j,2)}(\tau)\overline{\a_{1,+}^{(j,0)(\tau)}e^{-\iu\tau d}})
\\
=&\pi\left|\frac{\p\Psi_0^{(j)}}{\p x_1}(\g_0^-,\tau)- \frac{\p\Psi_0^{(j)}}{\p x_1}(\g_0^+,\tau)e^{-\iu\tau d}\right|^2.
\end{aligned}
\end{equation}
It follows from the above formula, (\ref{2.12}), and the unitarity of matrix $(a_{ij})$ that
\begin{equation}\label{4.45a}
k_2^{(1)}=\pi|L_\tau|^2\not=0,\quad k_2^{(2)}=0\quad \text{for each}\quad \tau\in\left[-\frac{\pi}{d},\frac{\pi}{d}\right).
\end{equation}

Function  $\Psi_2^{(j)}$ reads as
\begin{equation}\label{4.42a}
\Psi_2^{(j)}=\breve{\Psi}_2^{(j)} + C_1^{(j)}\Psi_0^{(1)}  + C_2^{(j)}\Psi_0^{(2)},
\end{equation}
where $\breve{\Psi}_2^{(j)}$ is the partial solution to (\ref{4.9})  fixed by the restriction
\begin{equation*}%\l%abel{4.43}
b_{1,2}^{(j,+)}\overline{b_{1,0}^{(p,+)}}  + b_{1,2}^{(j,-)}\overline{b_{1,0}^{(p,-)}}=0,\quad p,j=1,2,
\end{equation*}
and $C_1^{(j)}$, $C_2^{(j)}$ are arbitrary constants.
It follows from (\ref{4.42a}) that coefficients $\a_{2,\pm}^{(j,2)}$ in (\ref{4.11}) read as
\begin{equation}\label{4.47a}
\a_{2,\pm}^{(j,2)}=\breve{\a}_{2,\pm}^{(j,2)}  + C_1^{(j)} \a_{1,\pm}^{(j,0)} + C_2^{(j)}\a_{1,\pm}^{(j,0)},
\end{equation}
where coefficients $\breve{\a}_{2,\pm}^{(j,2)}$ come from asymptotics (\ref{4.11}) for $\breve{\Psi}_2^{(j)}$. In particular, it implies
\begin{gather}\label{4.47b}
\b_{3,\pm}^{(j,4)}=\breve{\b}_{3,\pm}^{(j,4)} + C_1^{(j)} \b_{1,\pm}^{(j,2)} + C_2^{(j)}\b_{1,\pm}^{(2,2)},
\\
\check{\b}_{3,-}^{(j,4)}=\frac{\breve{\a}_{2,-}^{(j,2)} +\breve{\a}_{2,+}^{(j,2)}e^{-\iu\tau d}}{2},\qquad \b_{3,+}^{(j,4)}=\frac{\breve{\a}_{2,+}^{(j,2)} +\breve{\a}_{2,-}^{(j,2)}e^{\iu\tau d}}{2}.
\end{gather}

The solvability of the problem for $\Psi_{4,1}^{(j)}$ is studied exactly as above and it yields the formula for $k_{4,1}^{(j)}$:
\begin{equation}\label{4.44}
k_{4,1}^{(j)}=\frac{3}{4}k_2^{(j)}.
%= 3\pi \left|\frac{\p\Psi_0^{(j)}}{\p x_1}(0)+ \frac{\p\Psi_0^{(j)}}{\p x_1}(\g_0^+)e^{-\iu\tau d}\right|^2.
\end{equation}

The solvability of problem (\ref{4.10}), (\ref{4.22}), (\ref{4.12b}) can be also studied as above; one just should introduce $\widetilde{\Psi}_4$ in an appropriate way, including in this function all the terms growing at infinity and all singular terms at $\g_0^\pm$. The solvability conditions of this problem and identities (\ref{4.21}), (\ref{4.33}), (\ref{4.34}), (\ref{4.35a}), (\ref{4.35b}), (\ref{4.37a}), (\ref{4.38a}), (\ref{4.47a}), (\ref{4.47b}) imply
\begin{equation}\label{4.45}
\begin{aligned}
k_4^{(j)}&\d_{jp} + k_2^{(j)} \big(C_1^{(j)}\d_{1p}+C_2^{(j)}\d_{2p}\big)
 \\
& + (k_2^{(j)})^2 \lim\limits_{R\to+\infty} \left(2\int\limits_{\Pi_R} \Psi_0^{(j)}(x)\overline{\Psi_0^{(p)}(x)}\di x - R(b_{1,0}^{(j,-)}\overline{b_{1,0}^{(p,-)}}-b_{1,0}^{(j,-)}\overline{b_{1,0}^{(p,-)}})
\right)
\\
=&\frac{\pi}{2} \left(8\breve{\a}_{3,-}^{(j,4)}\overline{\b_{1,-}^{(p,2)}} + 2 \b_{2,-}^{(j,2)}  \overline{\b_{2,-}^{(p,2)}} + 9 \b_{1,-}^{(j,2)}\overline{\b_{3,-}^{(p,2)}} + \frac{9(V(0)-\pi^2)}{4} \b_{1,-}^{(j,2)}\overline{\b_{1,-}^{(p,2)}}\right)
\\
&+4\pi (C_1^{(j)} \b_{1,-}^{(1,2)}\overline{\b_{1,-}^{(p,2)}} + C_2^{(j)} \b_{1,-}^{(2,2)}\overline{\b_{1,-}^{(p,2)}}), \quad j,p=1,2.
\end{aligned}
\end{equation}
We observe that, thanks to (\ref{2.12}), $\b_{1,-}^{(2,2)}=0$ for each $\tau\in\left[-\frac{\pi}{d},\frac{\pi}{d}\right)$. Due to this identity and (\ref{4.45a}) one can easily make sure that equations (\ref{4.45}) (w.r.t. $k_4^{(j)}$, $C_1^{(j)}$, $C_2^{(j)}$) are solvable once we let $C_1^{(1)}:=0$, $C_2^{(2)}:=0$. In particular, as $j=p=2$, we can determine $k_4^{(2)}$:
\begin{equation*}%\l%abel{4.50}
k_4^{(2)}=\pi|\b_{2,-}^{(2,2)}|^2.
\end{equation*}
Employing (\ref{4.35a}), (\ref{4.21}) and proceeding as in the proof of Lemma~3.1 in \cite{RJMP15}, we arrive at the final formula for $k_4^{(2)}$:
\begin{equation}\label{4.51}
k_4^{(2)}(\tau)= \frac{\pi}{4}\frac{\|L_\tau\|_{\mathds{C}^2}^2\|L'_\tau\|_{\mathds{C}^2}^2 - |(L_\tau,L'_\tau)_{\mathds{C}^2}|^2}{\|L_\tau\|_{\mathds{C}^2}^2}.
\end{equation}

\subsection{Justification}

In this section we discuss the justification of the asymptotics constructed formally in the previous section. As the justification, we mean proving estimates for the error terms. The scheme of the justification is borrowed from \cite[Sect. 4]{PMA15}. As in the previous section, we dwell on the case of two non-trivial solutions to problem (\ref{2.8}). The case of one solution can be treated in the same way.

First we observe that following the lines of the previous subsection, we can construct easily complete asymptotic expansions for $k_\e^{(j)}$, $j=1,2$, as well as for the associated nontrivial solutions of problem (\ref{2.8}).

Reproducing word-by-word the arguments of \cite[Subsect. 4.1]{PMA15}, one can make sure that these formal asymptotic series provide formal asymptotic solution to problem (\ref{4.7}).

At the next step we consider the boundary value problem
\begin{equation*}%\l%abel{4.46}
(-\D+V -\pi^2+k^2)u_\e=f\quad \text{in}\quad \Pi,\qquad u_\e=0\quad\text{on}\quad \p\Pi_\e,
\end{equation*}
with boundary conditions (\ref{2.3a}) and for $u_\e$ we assume behavior (\ref{4.2}) at infinity. Here $k$ is a complex parameter ranging in a small fixed neighborhood of zero, and $f$ is an arbitrary function in $L_2(\Pi)$ supported in $\overline{\Pi^0}$. As in Subsection~I\!V.A, the above problem can be rewritten to the operator equation
\begin{equation*}%\l%abel{4.52}
(I+\A_3(k,\e,\tau))g=f,
\end{equation*}
where $\A_3$ is the operator introduced in (\ref{4.48}). Proceeding as in \cite[Subsect. 4.2]{PMA15}, one can show that the solution to this equation obeys the representation:
\begin{equation*}%\l%abel{4.55}
g=\sum\limits_{j=1}^{2} Z_j \A_4\phi_0^{(j)} + \A_4 \A_2 f,
\end{equation*}
where
\begin{align*}
&\A_4:=(I+\A_2 \A_5)^{-1},\quad \A_5:=\A_3-\A_1,
\\
&
\begin{pmatrix}
Z_1
\\
Z_2
\end{pmatrix}=(kE+M)^{-1}
\begin{pmatrix}
F_1
\\
F_2
\end{pmatrix},\quad E:=
\begin{pmatrix}
1 & & 0
\\
0 &  &1
\end{pmatrix},
\\
&M:=
\begin{pmatrix}
&(\A_5 \A_4 \phi_0^{(1)},\psi_0^{(2)})_{L_2(\Pi^0)} &
(\A_5 \A_4 \phi_0^{(2)},\psi_0^{(1)})_{L_2(\Pi^0)}
\\
&(\A_5 \A_4 \phi_0^{(1)},\psi_0^{(2)})_{L_2(\Pi^0)} &
(\A_5 \A_4 \phi_0^{(2)},\psi_0^{(2)})_{L_2(\Pi^0)}
\end{pmatrix},
\\
&F_j:=(f,\psi_0^{(j)})_{L_2(\Pi^0)} - (\A_5 \A_4 \A_2 f,\psi_0^{(j)})_{L_2(\Pi^0)}.
\end{align*}
Employing this representation and following the lines of \cite[Subsect. 4.3]{PMA15}, one can get the desired estimates for the error terms.

As it was said in the beginning of Subsection~I\!V.B, if the real part of $k_\e^{(j)}$ is positive, it generates a discrete eigenvalue of $\He(\tau)$ by the rule $\l_\e^{(j)}=\pi^2-(k_\e^{(j)})^2$. It follows from (\ref{2.12}), (\ref{4.42}), (\ref{4.44}), (\ref{4.51}) that $\RE k_\e^{(j)}>0$, $j=1,2$, for sufficiently small $\e$ and hence, $\l_\e^{(j)}(\tau)=\pi^2-(k_\e^{(j)}(\tau))^2$ are the discrete eigenvalues of $\He(\tau)$ below the essential spectrum. In view of (\ref{4.8a}), the associated values $k_\e^{(j)}(\tau)$ do not coincide and hence, by Lemma~\ref{lm4.1}, both these eigenvalues are simple. The proof of Theorem~\ref{th2.3} is complete.

\begin{acknowledgments}

The work is partially supported by RFBR and by the grant of the President of Russia for young scientists-doctors of sciences (MD-183.2014.1).

\end{acknowledgments}

\end{document}